\documentclass{amsart}
\usepackage{graphicx,diagrams}
\usepackage{amsmath,amscd,amssymb,enumitem}%,accents}

\usepackage{tikz}

\newtheorem{thm}{Theorem}[section]
\newtheorem{cor}[thm]{Corollary}
\newtheorem{lem}[thm]{Lemma}
\newtheorem{prop}[thm]{Proposition}
\newtheorem*{main thm}{Main Theorem}
\newtheorem*{main lem}{Main Lemma}
\theoremstyle{definition}
\newtheorem{defn}[thm]{Definition}
\newtheorem{notn}[thm]{Notation}
\newtheorem{example}[thm]{Example}

\theoremstyle{remark}
\newtheorem{rem}[thm]{Remark}
\numberwithin{equation}{section}
\newcommand{\R}{\mathbb R}
\newcommand{\C}{\mathbb C} 
\newcommand{\Z}{\mathbb Z}

\newcommand{\Tc}{\mathbb T}

\newcommand{\tens}{\otimes} 
\newcommand{\dsum}{\oplus} 
\newcommand{\x}{\times}
 
\newcommand{\iso}{\simeq}

\renewcommand{\d}{\textrm{d}}
\renewcommand{\phi}{\varphi} 
\renewcommand{\to}{\longrightarrow}

\renewcommand{\mapsto}{\longmapsto}

\newcommand{\sr}{\mathcal}

\newcommand{\Aut}{\textrm{Aut}}
\newcommand{\Ann}{\textrm{Ann}} 
 
\newcommand{\del}{\partial}

\renewcommand{\^}{\wedge} 
 
\renewcommand{\Bar}{\overline}
\renewcommand{\epsilon}{\varepsilon}
\newcommand{\hide}[1]{}

\newcommand{\Id}{\textrm{Id}}

\renewcommand{\Im}{\textrm{Im}}

\begin{document}

\begin{abstract}
We answer the natural question: when is a transversely holomorphic symplectic foliation induced by a generalized complex structure?  The leafwise symplectic form and transverse complex structure determine an obstruction class in a certain cohomology, which vanishes if and only if our question has an affirmative answer.  We first study a component of this obstruction, which gives the condition that the leafwise cohomology class of the symplectic form must be transversely pluriharmonic.  As a consequence, under certain topological hypotheses, we infer that we actually have a symplectic fibre bundle over a complex base.  We then show how to compute the full obstruction via a spectral sequence.  We give various concrete necessary and sufficient conditions for the vanishing of the obstruction.  Throughout, we give examples to test the sharpness of these conditions, including a symplectic fibre bundle over a complex base which does not come from a generalized complex structure, and a regular generalized complex structure which is very unlike a symplectic fibre bundle, i.e., for which nearby leaves are not symplectomorphic.
\end{abstract}

\title{Symplectic foliations and generalized complex structures}
\author{Michael Bailey}
\email{bailey@cirget.ca}
\thanks{This work was done while the author was at the University of Toronto, and then at CIRGET/UQAM}
\maketitle

Generalized complex geometry (see, eg., \cite{Gualtieri2011} or \cite{Hitchin2003}) includes both symplectic and complex geometry as special cases.  In fact, near a \emph{regular point}, a generalized complex manifold ``looks like'' a product of a symplectic and complex manifold.  To be precise, a generalized complex structure induces a symplectic foliation (i.e., a Poisson structure) and a transverse complex structure; about a regular point, up to isomorphism, there is no more local information than this (and incidentally this is also true in the non-regular case, for subtle reasons).  Then is a regular generalized complex structure just the same as a transversely holomorphic symplectic structure?

In other words, given a regular Poisson structure $P$ and transverse complex structure $I$, it is natural to ask if $(P,I)$ are induced by a generalized complex structure.  This is the question we address in this paper.  The answer is always yes locally, so any obstruction must be global.  In fact, sometimes $(P,I)$ are \emph{not} generalized complex, as we shall see.  The obstruction places certain strong constraints on the relationship between the Poisson and transverse complex structures.

We use the method of coupling forms, whereby the leafwise symplectic form is extended to a 2-form on the whole manifold.  The calculation of the obstruction will then depend on this coupling form, with the complication that we must understand its interaction with the transverse complex structure.  For a good general reference for symplectic fibrations and coupling forms, see \cite{GuilleminLermanSternberg} or \cite[Chap. 6]{McDuffSalamon}.

\begin{samepage}
\subsection*{Summary}

\subsubsection*{Section \ref{intro section}}
We review the definitions and basic facts of generalized complex structures from the pure spinor viewpoint.
\end{samepage}

\subsubsection*{Section \ref{problem statement}}
We state our problem precisely, give the basic construction we will continue to use throughout the paper, and give some simple sufficient conditions for an affirmative answer to our question.  The construction always gives an \emph{almost} generalized complex structure; thus our concern is for its integrability.

\subsubsection*{Section \ref{integrability conditions}}
We study our construction in more detail, and give a necessary and sufficient condition for the existence of a compatible generalized complex structure (Theorem \ref{exact in truncated} and Proposition \ref{H iff alpha and beta}).  The condition takes the form of the vanishing of an obstruction class in a certain cohomology.  At this point, the condition is given relatively abstractly.

\subsubsection*{Section \ref{the pluriharmonic condition}}
We study one component of the obstruction, whose vanishing is both necessary and sufficient in certain low-dimensional cases. We consider \emph{smooth symplectic families}, i.e., symplectic foliations that come from a fibre bundle but which may not have local symplectic trivializations.  The condition is that the fibrewise symplectic form should be \emph{pluriharmonic} in the fibrewise cohomology bundle over the base (see Theorem \ref{[omega] pluriharmonic}, and Section \ref{the GM connection} for definitions).

For example, we consider a compact, connected smooth symplectic family with 2-dimensional fibres---or, in higher-rank cases, if certain topological conditions are satisfied---and conclude that if the data are generalized complex, it is in fact a symplectic fibre bundle.

(Counter)examples in this section include: admissible data which do \emph{not} come from a generalized complex structure; smooth symplectic families over a complex base which \emph{do} come from generalized complex structures, but which do \emph{not} have local symplectic trivializations (i.e., nearby leaves are not alike).

\subsubsection*{Section \ref{the full obstruction}}
We study the entire obstruction, and describe how to compute it, one component at a time, using a spectral sequence (see Theorem \ref{cohomology condition} for the statement).  For a symplectic fibre bundle over a complex base, two out of three components of the obstruction vanish, and the third involves a finite-dimensional calculation.  If the base is K\"ahler, the calculation simplifies somewhat.  However, even for a symplectic fibre bundle over a K\"ahler manifold, the third component of the obstruction will not always vanish, as in Example \ref{bad symplectic bundle}.

\subsection*{Acknowledgements}
This paper expands upon Chapter 3 in my Ph.D. thesis \cite{Bailey}.  During my time in graduate school at the University of Toronto, I received support, both material and mathematical, from many sources.  I would like to thank Yael Karshon, Brent Pym, Jordan Watts and Ida Bulat.  Particular thanks goes to Marco Gualtieri, with whom I had many useful discussions clarifying the ideas in this paper.

\section{Pure spinors and generalized complex structures}\label{intro section}

First we will recall some definitions and facts about transversely holomorphic foliations, then we will briefly review the pure spinor formalism of generalized complex structures.  This is not the usual way these structures are introduced---a generalized complex structure $J$ on a manifold $M$ is usually defined as a complex structure on a \emph{Courant algebroid} over $M$---but the data in either formalism determine each other, and in this paper we stick to only one for the sake of brevity.  For a thorough introduction, and proofs of claims in this section, see \cite{Gualtieri2011} (from which we have taken most of the material on generalized complex structures).

The important points to note are as follow: an (almost) generalized complex structure may be represented by its \emph{canonical line bundle} (Definition \ref{almost gc}) of pure spinors, which then has a pointwise decomposition into complex and symplectic parts, as in Proposition \ref{pointwise form of J}; the integrability condition, at least in the regular case which we are considering, amounts to the existence of closed (in a twisted sense) local generating sections for the canonical line bundle; the symmetries of generalized geometry extend the diffeomorphism group to include the $B$-transforms (Definition \ref{define B}); finally, for a regular, integrable generalized complex structure, the pointwise decomposition into complex and symplectic parts extends to a local normal form (Theorem \ref{regular local normal form}).

\begin{notn}
We indicate the complexification of a real vector bundle by a subscript $\C$, eg., $V_\C = \C \tens V \to M$.

In this paper we only consider smooth sections of bundles, which we denote $\Gamma(\cdot)$.  We let $i = \sqrt{-1}$ (except in those cases where $i$ denotes a degree).
\end{notn}

\subsection{Transversely holomorphic foliations}

In this section, we take our definitions from \cite{HaefligerSundararaman}.  We have the twin viewpoints of foliations either as Haefliger structures (from which it is straightforward to give meaning to ``transversely holomorphic'') or as integrable distributions (which is suitable for the decomposition of forms).

\begin{defn}A transversely holomorphic foliation of real codimension $2k$ on a manifold $M$ is given by an atlas of submersions $f_i : U_i \to \C^k$, for an open cover $\{U_i,\ldots\}$ of $M$, such that there exist holomorphic transition functions $\phi_{ij}$ between domains in $\C^k$, with $f_j = \phi_{ij} f_i$ on the intersections $U_i \cap U_j$.  Two transversely holomorphic foliations are equivalent if their atlases share a common refinement.
\end{defn}

This is equivalent to a real foliation $\sr{S}$ (whose leaves are preimages of points via $f_i$) and, on its normal bundle, a complex structure $I:N\sr{S} \to N\sr{S}$.  $I$ is a ``transverse structure'' in the sense that it is flat for the Bott connection of $\sr{S}$, and it is integrable in the sense that its Nijenhuis tensor vanishes.  These conditions may be combined with the integrability of $T\sr{S}$ into a single condition: 

\begin{prop}
A transversely holomorphic foliation is equivalent to a distribution $S \subset TM$ along with a complex structure $I$ on $NS := TM/S$, such that
\begin{equation}
S + N_{1,0}S
\end{equation}
is Lie-involutive.  ($S$ and $N_{1,0}S$ may be added by choosing any representatives for $N_{1,0}S$ in $T_\C M$, and $S + N_{1,0}S \subset T_\C M$ will not depend on this choice.)
\end{prop}
In other words, the integrability of $S + N_{1,0}S$ entails that $S$ is an integrable real distribution, and that $I$ is a transverse structure with vanishing Nijenhuis tensor.

\begin{rem}\label{yet another}
Alternatively, since $N^*S$ sits naturally in $T_\C M$ as $\Ann(S)$, yet another equivalent expression for the data of a transversely holomorphic foliation is just the \emph{canonical line bundle} of $I$,
\begin{equation}
\kappa_I := \^ ^k N^*_{1,0} S,
\end{equation}
which will have closed local generating sections.
\end{rem}

\subsection{Almost generalized complex structures}

\begin{defn}\label{define spinor}
By a \emph{spinor} on a manifold $M$ we will mean a (complex) mixed-degree differential form $\rho \in \Gamma(\^ ^\bullet T^*_\C M)$.
\end{defn}

\begin{defn}
Sections of $T_\C M \dsum T_\C^*M$ act on the spinors via the \emph{Clifford action}, by contraction and wedging: if $(X,\xi) \in \Gamma(T_\C M \dsum T_\C^*M)$ and $\rho$ is a spinor, then
$$(X,\xi)\cdot\rho = \iota_X \rho + \xi\^\rho.$$
Therefore, every spinor $\rho$ on $M$ has a \emph{null subbundle} $L_\rho \subset T_\C M \dsum T_\C^*M$ which is just its annihilator under the Clifford action.
\end{defn}

\begin{defn}
A spinor $\rho$ is \emph{pure} if $L_\rho$ is a \emph{maximal} isotropic subbundle with respect to the standard symmetric pairing on $T_\C M \dsum T_\C^*M$. 
\end{defn}
Such a maximal isotropic subbundle will have half the rank of $T_\C M \dsum T_\C^*M$; that is, its rank will be the dimension of $M$.

\begin{defn}
A (complex) maximal isotropic $L \subset T_\C M \dsum T^*_\C M$ has \emph{real rank zero} if $L \cap \bar{L} = 0$.  We will also say that a (complex) pure spinor $\rho$ has real rank zero if $L_\rho$ does.  (There is an alternative definition of real-rank-zero for spinors in terms of the Mukai pairing, which we omit.)
\end{defn}

An almost complex structure $I$ may be given by its canonical line bundle $\kappa_I$, the top wedge power of the $(1,0)$-forms.  Analogously,
\begin{defn}\label{almost gc}
An \emph{almost generalized complex structure} $J$ on $M$ is given by a pure spinor line bundle $\kappa_J \subset \^ ^\bullet T^*_\C M$ of real rank zero, called the \emph{canonical line bundle} of $J$.  (This includes almost complex structures, where $\kappa_J = \kappa_I$.)
\end{defn}

To each point $x \in M$ we associate a nonnegative integer, the \emph{type} of $J$ at $x$, which we may define as the lowest nontrivial degree of its canonical line bundle at $x$.

For example, in the case where an almost generalized complex structure is given by an almost complex structure, this is just the degree of $\kappa_I$, i.e., the complex dimension.  In general, we understand the type of $J$ as the number of complex (as opposed to symplectic) dimensions, as will be made clear (see Proposition \ref{pointwise form of J}).

\begin{defn}
$J$ is \emph{regular} at $x$ if its type is locally constant at $x$.
\end{defn}

\begin{rem}\label{formalism correspondence}
The alternative definition of generalized complex structures---as anti-involutions, $J:T M \dsum T^* M \to T M \dsum T^* M$, on the standard Courant algebroid---is related to this pure spinor formalism just by equating the $+i$-eigenbundle of $J$ with the null subbundle, $L_{\kappa_J}$, of the pure spinor bundle.
\end{rem}

%\begin{defn}\label{define transverse complex}
%Let $S \subset TM$ be a distribution on $M$, and let $I$ be an almost complex structure on $NS := TM/S$.  $I$ gives a decomposition of $N^*_\C S$ into $+i$ and $-i$--eigenbundles, $N^*_{1,0} S$ and $N^*_{0,1} S$ respectively.  The \emph{canonical line bundle} of $I$ is
%$$\kappa_I = \^ ^k N^*_{1,0} S,$$
%where $k = \dim_\C(N_{1,0} S)$.
%\end{defn}
%
%An almost complex structure $I$ on a normal bundle $NS$ is \emph{integrable} if $S + N_{1,0}S$ is Lie-involutive.  In this case, the data $I$, or $S + N_{1,0}$, or $\kappa_I$ are all equivalent to a transversely holomorphic foliation.

\begin{defn}\label{define exponential}
We define the \emph{Clifford exponential} via the usual Taylor series. In particular, if $B$ is a form, then
$$e^B = 1 + B + \frac{1}{2} B \^ B + \ldots,$$
where in this case the products eventually vanish and the series is finite.  Note that, for forms of even degree, this exponential is a homomorphism from $+$ to $\^$.
\end{defn}

\subsubsection*{Pointwise decomposition}
At a point, i.e., on a single tangent space, there is a characterization of (almost) generalized complex structures in terms of a decomposition into complementary ``symplectic leaf'' and ``transverse complex'' parts, possibly modified by a $B$-transform (see Definition \ref{define B}), as follows:
\begin{prop}\label{pointwise form of J}
If $J|_x$ is an (almost) generalized complex structure \emph{at a point} $x$, then there exist
\begin{itemize}
\item a subspace $S_x \subset T_x M$,
\item an (almost) complex structure $I$ on the normal space $N_x S$,
\item and real 2-forms $\omega$ and $B$ in $\^ ^2 T^*_x M$---where $\omega$ is nondegenerate when pulled back to $S_x$
\end{itemize}
such that the canonical line bundle at $x$ is of the form
\begin{equation}
\left.\kappa_J\right|_x = e^{B+i\omega} \^ \kappa_I
\end{equation}
\end{prop}

If $J$ is \emph{regular} at $x$ (as will always be the case in this paper), then such a representation extends to a neighbourhood.  Furthermore, any spinor line bundle of this form determines an almost generalized complex structure.

\begin{rem}\label{S and I well-defined}
The lowest-degree component of $\kappa_J$ is just $\kappa_I$; hence, the distribution $S = \Ann(\kappa_I)$ and the almost complex structure $I$ on $NS$ are uniquely determined by $J$.  However, $B + i\omega$ is not.  Rather, $B+i\omega$ is well-defined up to $N^*_{1,0} \^ T^*_\C M$; that is, $B+i\omega$ is a well-defined section of $\^ ^2 \left(S_\C + N_{0,1}S\right)^*$.
\end{rem}

In generalized geometry, the symmetry group of a manifold is understood to be an extension of the diffeomorphisms.  In addition, it includes the $B$-transforms:
\begin{defn}\label{define B}
If $B$ is a real 2-form and $J$ is an almost generalized complex structure with canonical line bundle $\kappa_J$, then the \emph{$B$-transform} of $J$ is written $e^B\cdot J$, and may be defined in terms of its action on $\kappa_J$:
$$\kappa_{e^B\cdot J} = e^B \^ \kappa_J.$$
\end{defn}
We distinguish between closed $B$-transforms and non-closed $B$-transforms, since when the 2-form $B$ is non-closed, the integrability condition changes (see Proposition \ref{H+dB}).

\subsection{Integrability of generalized complex structures}

\begin{defn}\label{define integrability}
If $H$ is a closed real 3-form and $\rho$ is a spinor, then
$$d_H \rho := d\rho + H\^\rho.$$
We say that a pure spinor $\rho$ is \emph{$H$-integrable} if
$$d_H \rho = (X,\xi)\cdot\rho$$
for some $(X,\xi) \in \Gamma(T_\C M \dsum T^*_\C M)$.

We say that an almost generalized complex structure $J$ is $H$-integrable, or, alternatively, that $J$ is a generalized complex structure with curvature $H$, if every local section of the canonical line bundle $\kappa_J$ of $J$ is $H$-integrable.
\end{defn}
\begin{rem}
In fact, if about every point in $M$ there is \emph{some} integrable local generating section of $\kappa_J$, then this is sufficient: then \emph{every} local section of $\kappa_J$ will be integrable and thus $J$ will be integrable.

Furthermore, if $J$ is $H$-integrable and is \emph{regular} at $x$, then in fact there is a $d_H$-closed local generating section of $\kappa_J$ about $x$.
\end{rem}

\begin{prop}\label{H+dB}
If $J$ is an $H$-integrable generalized complex structure for some closed 3-form $H$, and $B$ is a 2-form, then $e^B\cdot J$ is an $(H + dB)$-integrable generalized complex structure. 
\end{prop}

\begin{thm}[Gualtieri, \cite{Gualtieri2011}]\label{regular local normal form}
If $J$ is a generalized complex structure regular at $x$, then there is a neighbourhood of $x$ which is isomorphic---via diffeomorphism and $B$-transform---to a neighbourhood in $\R^{n-2k} \times \C^k$, with generalized complex structure (integrable for $H=0$) given by
$$\kappa = e^{i\omega} \^ \kappa_I,$$
where $\omega$ is the standard symplectic form on $\R^{n-2k}$ and $\kappa_I$ is the canonical bundle for the complex structure on $\C^k$.
\end{thm}
In other words, near a regular point, any generalized complex structure is equivalent to the product of a complex structure with a symplectic structure.  (In the vicinity of type change the situation is different, and has been studied in \cite{AbouzaidBoyarchenko} and \cite{Bailey}.)

The following is an easy corollary of Theorem \ref{regular local normal form}:

\begin{prop}\label{J induces P and I}
Let $J$ be a (regular) generalized complex structure integrable with respect to some closed 3-form, and let
$$\kappa_J = e^{B+i\omega} \^ \kappa_I$$
be the local form decomposition of its canonical line bundle.  Then $\kappa_I$ determines a transversely holomorphic foliation, as in Remark \ref{yet another}, and $\omega$ pulls back to a leafwise symplectic form.
\end{prop}

Thus we say that $J$ determines a transversely holomorphic symplectic foliation.  The real symplectic foliation is equivalent to a Poisson structure, $P$, and the transversely holomorphic structure is just a complex structure $I$ on the normal bundle to $P$'s symplectic leaves, thus we will often indicate the transversely holomorphic symplectic foliation with the data $(P,I)$.

\section{Problem statement and non-integrable solution}\label{problem statement}

The precise meaning of our central question is now clear:
\begin{itemize}
\item When does a transversely holomorphic symplectic foliation come from a generalized complex structure on $M$, as in Proposition \ref{J induces P and I}?
\end{itemize}
This is always the case locally, as an easy consequence of Corollary \ref{complementary foliation}, so any obstruction must be global.  The global answer is ``not always.''

However, we can always find \emph{almost} generalized complex structures inducing a given transversely holomorphic symplectic foliation.  We give a construction, which will then be the basis of our general solution of the central question.

\begin{notn}
In what follows, let $P$ be a regular Poisson structure and $I$ a complex structure transverse to its leaves.  (As we remarked, the data $(P,I)$ are equivalent to a transversely holomorphic symplectic foliation.)  Let $\sr{S}$ be the symplectic foliation of $P$, $S = T\sr{S}$ be its tangent distribution, and let $\omega \in \Gamma(\^ ^2 S^*)$ be the induced leafwise symplectic form.  ($\omega$ is leafwise-closed, and leafwise nondegenerate.)
\end{notn}

\begin{rem}
We say that the data $(P,I)$ ``are'' generalized complex as shorthand to mean they are induced by a generalized complex structure.
\end{rem}

\begin{notn}
We will use $NS := TM/S$ to indicate the normal bundle to $S$, whereas $N \subset TM$ will refer to a particular choice of distribution complementary to $S$, i.e., representatives of $NS$.  Then the transverse complex structure $I$ induces an almost complex structure on $N$, splitting it as $N_\C = N_{1,0} \dsum N_{0,1}$.
\end{notn}

\begin{defn}\label{definition from N}
If $N \subset TM$ is a smooth distribution complementary to $S$, then we may extend $\omega$ to $M$ by specifying $\ker(\omega) = N$.

Then the almost generalized complex structure, $J_N$, induced by $(P,I)$ and the choice of $N$, is defined by the canonical line bundle
\begin{equation}\label{defining spinor}
\kappa = e^{i\omega} \^ (\^ ^k N_{1,0}^*) = e^{i\omega} \^ \kappa_I
\end{equation}
\end{defn}

This construction is general:

\begin{prop}\label{construction is general}
Let $J$ be an almost generalized complex structure inducing $(P,I)$.  Then for any choice of complementary distribution $N \subset TM$, $J$ is equivalent, up to a $B$-transform, to the structure $J_N$.
\end{prop}
\begin{proof}
As per Proposition \ref{pointwise form of J}, the canonical line bundle of $J$ is
$$\kappa_J = e^{B+i\omega} \^ \kappa_I,$$
where $B+i\omega$ is a section of $\^ ^2 \left(S_\C \dsum N_{0,1}\right)^*$.  By choosing $N \subset TM$ and specifying $B+i\omega \in \Ann(N_{1,0})$, we extend $B+i\omega$ to a section of $\^ ^2 T^*_\C M$.  $\omega$ extended in this way will be the same $\omega$ as in Definition \ref{definition from N} above.  Then
$$e^{-B} \^ \kappa_J = e^{i\omega} \^ \kappa_I.$$
\end{proof}

\subsection{Easy case}

Using Definition \ref{definition from N}, we can answer our question in the affirmitive in some special cases.  The following was originally observed by Cavalcanti \cite{Cavalcanti}:

\begin{prop}
If the leafwise-closed symplectic form $\omega$ extends to a closed form $\tilde\omega$ on $M$, then $(P,I)$ are generalized complex, integrable with curvature $H=0$.
\end{prop}
\begin{proof}
The generalized complex structure is defined by the canonical line bundle
$$e^{i\tilde\omega} \^ \kappa_I,$$
which admits closed local sections.
\end{proof}

\begin{cor}\label{complementary foliation}
If $\sr{S}$ admits a complementary foliation $\sr{R}$, for which $\omega$ is constant in the directions of $R$ (i.e., if $\nabla^\sr{R}$ is the Bott connection of $R$, then $\nabla^\sr{R} \omega = 0$), then $(P,I)$ are generalized complex.
\end{cor}
\begin{proof}
We choose the complementary foliation $N = T\sr{R} \subset TM$.  As in Definition \ref{definition from N}, we extend $\omega$ to $TM$ by specifying that $\ker(\omega) = N$.  Since $\omega$ is constant along the directions of $\sr{R}$, and $\omega$ is leafwise closed on $\sr{S}$, then on the total space $d\omega=0$.  (We can see this by expressing a neighbourhood as a product decomposition for $\sr{S}$ and $\sr{R}$.)
\end{proof}

\subsection{Smooth symplectic families}

It is clarifying to consider our problem in the case where the symplectic foliation actually comes from a fibre bundle.  A transversely holomorphic symplectic foliation whose leaf space is smooth is the same thing as a fibre bundle with complex base and symplectic fibres.  Note, however, that we are not talking here about ``symplectic fibre bundles,'' since in this case there may not be local \emph{symplectic} trivializations.

\begin{defn}\label{define smooth symplectic family}
A \emph{smooth symplectic family} over a complex manifold $B$ is a fibre bundle $\pi:X \to B$ with a Poisson structure whose symplectic leaves coincide with the fibres.  By pullback from $B$, it inherits a complex structure transverse to the symplectic foliation.

We say that a smooth symplectic family over a complex manifold is \emph{generalized complex} if its Poisson structure and transverse complex structure are induced by a generalized complex structure.
\end{defn}

In this case, choosing a complementary distribution $N$ is just choosing the horizontal distribution of a connection.  Integrability of $N$ corresponds to flatness of the connection.

\begin{example}
As we remarked, smooth symplectic families are not in general symplectic fibre bundles; however, if $X \to B$ \emph{is} a symplectic fibre bundle over a complex base $B$, with flat symplectic connection (i.e., one which preserves the symplectic form), then it will be generalized complex, since the complex structure on $B$ pulls back to a transverse complex structure on $X$, and the flat symplectic connection gives us precisely the complementary foliation needed in Corollary \ref{complementary foliation}.
\end{example}

\section{Integrability---the general case}\label{integrability conditions}

If $J$ is a generalized complex structure inducing $(P,I)$, then for \emph{any} choice of $N$ complementary to $S$, $J$ is equivalent via a $B$-transform to the almost generalized complex structure $J_N$ in definition \ref{definition from N}.  But then $J_N$ is integrable---if $J$ was $H'$--integrable and $e^B\cdot J = J_N$, then $J_N$ is $H$--integrable for $H=H'+dB$.  Contrapositively, if $J_N$ is not integrable, then there is no generalized complex structure $J$ inducing $(P,I)$.

Therefore, we answer the question of whether $(P,I)$ comes from a generalized complex structure by choosing any complementary $N$, and then testing to see if $J_N$ is integrable for some closed 3-form $H$.

Fix the choice $N \subset TM$ complementary to $S$.  Suppose that $H$ is a closed 3-form such that $J_N$ is $H$-integrable.  We will study two types of conditions on $H$---the $H$-integrability of $J_N$, which gives equations relating $H$ to the symplectic form $\omega$, and the closedness condition  $dH=0$.  Taken together, these give a cohomological condition (Theorem \ref{exact in truncated}) on a quantity derived from $\omega$, i.e., an obstruction map.

We will decompose the obstruction according to a grading (Proposition \ref{H iff alpha and beta}).  One component of the obstruction, which is both necessary and sufficient in certain low-dimensional cases, is easier to understand than the others, and we give it special attention in Section \ref{the pluriharmonic condition}.  Later, we compute the remaining components of the obstruction by using a spectral sequence (summarized in Theorem \ref{cohomology condition}).
 
\subsection{Trigrading and the decomposition of $d$}\label{trigrading}

\begin{notn}
We use $\Omega^\bullet(M)$ to denote the smooth differential forms on $M$.  Sometimes we indicate a degree in place of $\bullet$, and sometimes we omit $M$ when it is clear which manifold we are considering.

The symplectic foliation along with the choice of $N$ allow us to decompose forms in a couple of ways.  We use $(j;k)$ (separated by a semicolon) to indicate degree $j$ in $N^*$ and degree $k$ in $S^*$.  But $\^ ^\bullet N^*$ further decomposes by the complex bigrading.  We use $(i,j)$ (separated by a comma) to indicate a complex bidegree, and $(i,j;k)$ to indicate a tri-degree in $N^*_{1,0}$, $N^*_{0,1}$ and $S^*$ respectively.

For example,
$$\Omega^{n;k} = \bigoplus_{i+j=k} \Omega^{i,j;k}.$$
\end{notn}

\begin{rem}
The $(n;k)$ bigrading---without the further decomposition into the trigrading---often appears in the theory of coupling forms and symplectic fibrations (see, eg., \cite{CrainicFernandes}).
\end{rem}

If $N$ were integrable, then we could decompose the exterior derivative as $d = \nabla + d_S$, where $\nabla$ had degree $(1;0)$ and $d_S$ had degree $(0;1)$.  Furthermore, $\nabla$ would decompose as $\del + \bar\del$ according to the transverse complex structure.  Each of these components would square to zero and we would have a triple complex.  However, since $N$ may not be integrable, we will have an additional curvature term.

\newdiagramgrid{spectral}%
{1,1,1,1,1,1,1,1,1,1,1}%
{1,.5,.5,.5,.5,.5,.5,.5,.5,.5,.5,.5,.5}

\begin{figure}
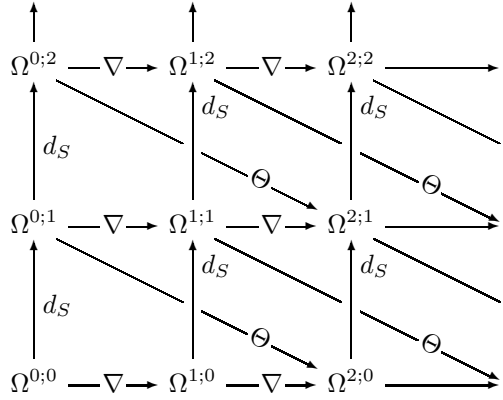
\label{spectral fig}
\begin{diagram}[grid=spectral]
%{} & {} & {} & {} & {} & {} & {} & {} & {} \\
\uTo &  & \uTo & & \uTo & {} & {} \\%& {} & {} \\
\Omega^{0;2}  &  \rTo~\nabla  &  \Omega^{1;2}  &  \rTo~\nabla  &  \Omega^{2;2}  &  \rTo  &  {} \\%& {} & {}  \\
  &  \rdLine(2,2)  &  \uTo>{d_S}  &  \rdLine(2,2)  &  \uTo>{d_S}  &  \rdLine(2,2)  &  {} \\%& {} & {}  \\
\uTo>{d_S}  &    &    &    &    &   &  {} \\%& {} & {}  \\
  &    &    &  \rdTo(2,2)~\Theta  &    &  \rdTo(2,2)~\Theta  &  {} \\%& {} & {}  \\
\Omega^{0;1}  &  \rTo~\nabla  &  \Omega^{1;1}  &  \rTo~\nabla  &  \Omega^{2;1}  &  \rTo  &  {} \\%& {} & {}  \\
  &  \rdLine(2,2)  &  \uTo>{d_S}  &  \rdLine(2,2)  &  \uTo>{d_S}  &  \rdLine(2,2)  &  {}\\% & {} & {}  \\
\uTo>{d_S}  &    &    &    &    &   &  {}\\% & {} & {}  \\
  &    &    &  \rdTo(2,2)~\Theta  &    &  \rdTo(2,2)~\Theta  &  {}\\% & {} & {}  \\
\Omega^{0;0}  &  \rTo~\nabla  &  \Omega^{1;0}  &  \rTo~\nabla  &  \Omega^{2;0}  &  \rTo  &  {}\\% & {} & {} \\
%{} & {} & {} & {} & {} & {} & {} & {} & {} \\
%{} & {} & {} & {} & {} & {} & {} & {} & {} 
\end{diagram}
\caption{A diagram showing the action of each real component of $d$ on the real $(j;k)$ bigrading.  It suggests a spectral sequence with three successive differentials lying on a diagonal.}
\end{figure}

\begin{lem}\label{d decomposition}
$d$ decomposes as
\begin{equation}\label{d formula}
d = \underset{1;0}{\nabla} + \underset{2;-1}{\Theta} + \underset{0;1}{d_S}.
\end{equation}
(Under each term we indicate the degree of the operator.)  $\Theta$ acts as a tensor in \mbox{$\Gamma(\^ ^2 N^* \tens S)$,} by contracting in $S$ and wedging in $N^*$.

Furthermore, over the complex forms $\nabla$ decomposes as
$$\nabla = \underset{1,0;0}{\del} + \underset{0,1;0}{\bar\del}$$
and $\Theta$ as
$$\Theta = \underset{2,0;-1}{\theta_+} + \underset{1,1;-1}{\theta_0} + \underset{0,2;-1}{\theta_-}.$$
$\del$ and $\bar\del$ are complex-conjugate, as are $\theta_+$ and $\theta_-$, and $\theta_0$ is real.
\end{lem}

We omit a proof of Lemma \ref{d decomposition}, since similar results are found in the literature.  See, for example, a real version in \cite[Proposition 10.1]{BGV1992}.

In the case of a fibre bundle with connection, $\nabla$ is just the covariant derivative and $\Theta$ is its curvature.  Figure 1 shows the action of each real component of $d$ acting on the $(j;k)$ bigrading.

\begin{rem}\label{quadratic relations}
It is not the case that each term in the decomposition of $d$ squares to zero; but, of course, $d^2=0$, and by decomposing this equation according to degree we may find quadratic relations between terms.  In the full complex trigrading, these relations may be summarized thus: the terms other than $\del$, $\bar\del$ and $\theta_0$ \emph{do} square to zero, and the terms pairwise anticommute, with the exception of those pairs occuring in the following special anticommutation relations:
\begin{eqnarray}
\del^2 + d_S \theta_+ + \theta_+ d_S &=& 0 \qquad \textnormal{(and the conjugate relation)} \label{AC1} \\
\del\bar\del + \bar\del\del + \theta_0 d_S + d_S \theta_0 &=& 0 \label{AC2} \\
\theta_+\theta_- + \theta_-\theta_+ + \theta_0^2 &=& 0 \label{AC3} \\
\theta_+\bar\del + \bar\del\theta_+ + \theta_0\del + \del\theta_0 &=& 0 \qquad \textnormal{(and the conjugate relation)} \label{AC4}
\end{eqnarray}
Analogous but simpler relations hold between the real counterparts $\nabla$, $\Theta$ and $d_S$.
\end{rem}

\subsection{Truncated complex}\label{truncated complex}
We give one further decomposition which will be useful.

\begin{defn}
The \emph{truncated de Rham complex} $\Omega^\bullet_\Tc$ consists of the forms
$$\Omega^{1,0;0} \^ \Omega^\bullet = \bigoplus_{i\geq1}\Omega^{i,\bullet;\bullet},$$
that is, those forms which have at least one degree in $N^*_{1,0}$.  The \emph{real} truncated de Rham complex, $\Omega^\bullet_T$, consists of the real forms in $\Omega^\bullet_\Tc$; in other words, those (real) forms which have at least one degree in \emph{both} $N^*_{1,0}$ and $N^*_{0,1}$.  As we see in Figure 2, this omits the ``outside edges'' of the complex bigrading.
\end{defn}

Note that $\Omega^\bullet_T$ and $\Omega^\bullet_\Tc$ are each differential complexes for the full exterior derivative $d$.  

\begin{figure}
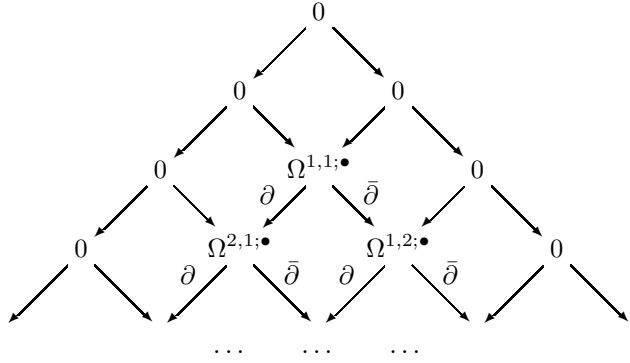

\begin{diagram}[height=1.5em,width=1.5em,nohug]
& & & & & & & & 0 & & & & & & & & \\
& & & & & & & \ldTo & & \rdTo & & & & & & & \\
& & & & & & 0 & & & & 0 & & & & & & \\
& & & & & \ldTo & & \rdTo & & \ldTo & & \rdTo & & & & & \\
& & & & 0 & & & & \Omega^{1,1;\bullet} & & & & 0 & & & & \\
& & & \ldTo & & \rdTo & & \ldTo^{\del} & & \rdTo^{\bar\del} & & \ldTo & & \rdTo & & & \\
& & 0 & & & & \Omega^{2,1;\bullet} & & & & \Omega^{1,2;\bullet} & & & & 0 & & \\
& \ldTo & & \rdTo & & \ldTo^{\del} & & \rdTo^{\bar\del} & & \ldTo^{\del} & & \rdTo^{\bar\del} & & \ldTo & & \rdTo & \\
& & & & & & & & & & & & & & & & 
\end{diagram}
\ldots \qquad \ldots \qquad \ldots
\caption{The real truncated de Rham complex $\Omega^\bullet_T$.}
\label{fig:differentials}
\end{figure}

Recall that $K = N_{0,1} \dsum S_\C$ is the distribution whose involutivity determines the integrability of the transverse complex structure.  We may decompose the forms as
\begin{eqnarray*}
\Omega^\bullet &=& \Omega^{0,\bullet;\bullet} \dsum \left(\Omega^{1,0;0} \^ \Omega^\bullet\right) \\
&=& \Gamma(\^ ^\bullet K^*) \dsum \Omega^\bullet_\Tc.
\end{eqnarray*}
If $\alpha$ is a form, we denote projection to the first summand in this decomposition by $\alpha|_K$.  We define an operator $d_K$ such that
$$d_K\alpha = (d\alpha)|_K.$$

\subsection{Relation of $H$ to $\omega$}

\begin{prop}\label{H constraint lemma}
Let $H$ be a real closed 3-form.  Then $J_N$, as in Definition \ref{definition from N}, is $H$-integrable if and only if
\begin{equation}
H|_K + id_K\omega = 0. \label{constraint on H}
\end{equation}
Equivalently, $J_N$ is $H$-integrable if  and only if the following equations hold:
\begin{eqnarray}
H^{0,0;3} &=& 0, \label{H1} \\
H^{0,1;2} &=& -i\bar\del\omega, \label{H2} \\
H^{0,2;1} &=& -i\theta_-\omega \quad\textnormal{and} \label{H3} \\
H^{0,3;0} &=& 0. \label{H4}
\end{eqnarray}
\end{prop}

\begin{proof}
Locally, the canonical bundle of $J_N$ is generated by a section $\rho \^ e^{i\omega}$, where $\rho \in \Gamma(\kappa_I)$ is some local, closed generator for the canonical bundle of the transverse complex structure.  The $H$-integrability of $J_N$ says that
\begin{eqnarray*}
0 &=& d_H (\rho \^ e^{i\omega}) \\
&=& \rho \^ (d_H e^{i\omega}) \\
&=& \rho \^ (d + H\^) e^{i\omega} \\
\Leftrightarrow 0 &=& \rho \^ (i d \omega + H)
\end{eqnarray*}
In this equation, any component of $id\omega+H$ that has nonzero degree in $N_{1,0}^*$ annihilates with $\rho$ and provides no constraint; thus we have Equation \eqref{constraint on H}.  Looking in each degree, for all $j$ and $k$ we should have
$$-i(d \omega)^{0,j;k} = H^{0,j;k}.$$

Equations \eqref{H2}, \eqref{H3} and \eqref{H4} are just selected degrees of this condition, according to the trigrading.  The $(0,0;3)$-degree component of this condition is $H^{0,0;3} = -i d_S \omega$; however, we supposed that $\omega$ was leafwise-closed, i.e., $d_S\omega=0$, and so we get Equation \eqref{H1}.
\end{proof}

Since $H$ is real, we have that $H^{i,j;k} = \overline{H}^{j,i;k}$, thus these equations also determine $H^{1,0;2}$, $H^{2,0;1}$ and $H^{3,0;0}$. There remain two free terms, $H^{1,1;1}$ and $H^{2,1;0}$.  ($H^{1,2;0}$ must be conjugate to $H^{2,1;0}$.)

\subsection{Closedness of $H$}

\begin{lem}
$d_K^2=0$, and if $\alpha$ is a form then $dd_K\alpha \in \Omega^\bullet_T$.
\end{lem}
\begin{proof}
That $d_K^2=0$ follows from the vanishing of the $(0,j;k)$ degree of $d^2=0$. But
$$dd_K\alpha = (\del + \theta_+ + \theta_0) d_K \alpha + \d_K \d_K \alpha$$
where the first summand has positive degree in $N_{1,0}^*$, and thus the second claim follows.
\end{proof}

\begin{thm}\label{exact in truncated}
Suppose that $P$ and $I$ are a regular Poisson structure and a transverse complex structure respectively.  Let $N$ be a choice of complementary distribution to the symplectic distribution $S$, let $\omega$ be the thusly--extended coupling symplectic form (as in Definition \ref{definition from N}), and let $d_K$ be as in Section \ref{truncated complex}.

Then $(P,I)$ comes from a generalized complex structure if and only if the imaginary part of $dd_K\omega$ is exact \emph{in the truncated complex} $\Omega^\bullet_T$.
\end{thm}

Since the imaginary part $\Im(dd_K\omega)$ is real and of degree $4$, the condition of this theorem means it must equal $d\sigma$ for some real $\sigma$ of degrees $(2,1;0) + (1,2;0) + (1,1;1)$.

\begin{proof}
We take the almost generalized complex structure $J_N$ as in Definition \ref{definition from N}, and try to find a real closed 3-form $H$ integrating it.  This succeeds if and only if $(P,I)$ comes from a generalized complex structure.  Such an $H$, if it exists, is determined by Equation \eqref{constraint on H}, the reality condition and the free terms $H^{1,1;1}$ and $H^{2,1;0}=\Bar{H^{1,2;0}}$.  Furthermore, $dH$ must vanish.

We will decompose $H$ according to its component in $\^ ^3 K^*$, its component in $\^ ^3 \Bar{K^*}$, and its component in the truncated complex.  $H^{0,0;3}$ is the component of $H$ lying in both $\^ ^3 K^*$ and $\^ ^3 \Bar{K^*}$.  Since $H^{0,0;3}=0$ (by Equation \eqref{H1}), we have that
\begin{eqnarray*}
H &=& H + H^{0,0;3} \\
&=& H|_K + \overline{H|_K} + H^{2,1;0} + H^{1,2;0} + H^{1,1;1} \\
&=& -id_K\omega + \overline{-id_K\omega} + H^{2,1;0} + H^{1,2;0} + H^{1,1;1} \qquad \textnormal{(by Equation \eqref{constraint on H})} \\
&=& 2\Im(d_K\omega) + H^{2,1;0} + H^{1,2;0} + H^{1,1;1}
\end{eqnarray*}

Then $dH=0$ if and only if
$$-2\Im(dd_K\omega)=d(H^{2,1;0} + H^{1,2;0} + H^{1,1;1}).$$
That is, $\Im(dd_K\omega)$ should be exact in the (real) truncated complex.
\end{proof}

For concreteness, we write the condition as a system of equations in two free components, grouped according to the degree in $S^*$.  Figure ~\ref{fig:obstruction} illustrates this system in a diagram.

\begin{prop}\label{H iff alpha and beta}
In the setup of Theorem \ref{exact in truncated}, with $d$ decomposed as in Lemma \ref{d decomposition}, $(P,I)$ comes from a generalized complex structure if and only if there exist real $3$-forms
$$\alpha \in \Omega^{1,1;1} \;\textnormal{and}\; \beta \in \Omega^{2,1;0}+\Omega^{1,2;0}$$
such that the following hold:
$$\begin{array}{rcll}
i(\del\bar\del-\bar\del\del)\omega &=& d_S\alpha \quad& (A) : (2;2)\\
i\left(\nabla(\theta_- - \theta_+) + \Theta(\bar\del - \del)\right)\omega
&=& \nabla\alpha + d_S\beta \quad& (B) : (3;1)\\
i\Theta(\theta_- - \theta_+)\omega &=& \Theta\alpha + \nabla\beta \quad& (C) : (4;0)
\end{array}$$
\end{prop}

(To the right of each equation we have given an alphabetic label and indicated its degree.)

\begin{proof}
Let $\alpha = H^{1,1;1}$ and $\beta = H^{2,1;0}+H^{1,2;0}$.  From the proof of Theorem \ref{exact in truncated}, we have the condition
$$-2\Im(dd_K\omega) = d(H^{2,1;0} + H^{1,2;0} + H^{1,1;1}) = d(\beta + \alpha).$$
If we write down the nontrivial degrees of this condition, i.e., those where $\alpha$ or $\beta$ occurs, and then make substitutions for $d_K\omega$ according to Proposition \ref{H constraint lemma}, we get equations (A), (B) and (C).
\end{proof}

Throughout this paper, we will refer to these equations as ``conditions (A), (B) and (C).''

\begin{defn}\label{obstructions}
We define the \emph{obstruction form} $\Phi(\omega)=-2\Im(dd_K\omega)$, and its components
\begin{align*}
\Phi^A(\omega) &= i(\del\bar\del-\bar\del\del)\omega \\
\Phi^B(\omega) &= i\left(\nabla(\theta_- - \theta_+) + \Theta(\bar\del - \del)\right)\omega \\
\Phi^C(\omega) &= i\Theta(\theta_- - \theta_+)\omega
\end{align*}
\end{defn}
Then, as remarked, the vanishing of $\Phi(\omega)$ in the truncated cohomology corresponds to the existence of an appropriate generalized complex structure.

\begin{figure}
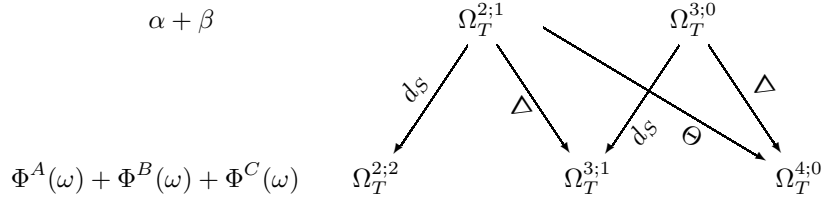

\begin{diagram}[width=2em]
\alpha + \beta  & & & \Omega_T^{2;1} & & & & \Omega_T^{3;0} & & \\
& & \ldTo^{d_S} & & \rdTo<\nabla & \rdTo(5,2)_{{}\qquad\qquad\Theta} & \ldTo_{d_S\qquad{}} & & \rdTo^\nabla & \\
\Phi^A(\omega)+\Phi^B(\omega)+\Phi^C(\omega) \qquad & \Omega_T^{2;2} & & & & \Omega_T^{3;1} & & & & \Omega_T^{4;0}
\end{diagram}
\caption{A diagram of part of the truncated complex. The obstruction form, $\Phi(\omega)$, lives in the lower row, and for the data to be generalized complex it must have a primitive in the upper row.}
\label{fig:obstruction}
\end{figure}

\begin{rem}\label{type 1}
If the type---that is, the complex dimension of $N_{1,0}$---is 1, then conditions (B) and (C) are trivial.  Condition (A) has a simple interpretation, as we shall see.
\end{rem}

\section{The pluriharmonic condition}\label{the pluriharmonic condition}

Condition (A) is just a condition on the de Rham cohomology class of $\omega$ on the leaves.  To make this clear, we turn our attention to the case of a smooth symplectic family $X \to B$, where the fibrewise cohomology can be understood as a vector bundle over $B$.

\subsection{The Gauss-Manin connection}\label{the GM connection}

As we said, the operator $\nabla = \del + \bar\del$ is just the covariant derivative of a chosen bundle connection, and $\Theta = \theta_+ + \theta_0 + \theta_-$ is its curvature.  Similarly to Remark \ref{quadratic relations}, we have the following quadratic relation:
$$\nabla^2 = -d_S\Theta - \Theta d_S$$
So $\nabla^2$ vanishes up to $d_S$-chain-homotopy.  In other words, $\nabla$ determines a \emph{flat} connection on the fibrewise cohomology bundle.  This is known as the \emph{Gauss-Manin connection}, and it is canonical---that is, it doesn't depend on the choice of $N$.  We recall a connection-independent definition, just to make this fact clear.  For details, and a somewhat more general presentation, see \cite{KatzOda}. 

\begin{rem}
We describe the real case first, and then remark on how it complexifies.
\end{rem}

\subsection*{Relative cohomology}
Given a choice of horizontal distribution $N$, we had a bi- or trigrading on forms.  Absent such a choice, we can still define an associated graded object, which will be (non-canonically) equivalent to the graded forms.

On a fibre bundle $\pi:X \to B$, there is a natural filtration of the forms,
$$\Omega^\bullet(X) = F_0^\bullet \supset F_1^\bullet \supset F_2^\bullet \supset \ldots,$$
where $F_n^m$ consists of the $m$-forms on $X$ generated (over $\Omega^\bullet(X)$) by pullbacks of $n$-forms on $B$.  Let $\Lambda^{n;k}(X) = F_n^{n+k}/F_{n+1}^{n+k}$ be the associated graded object.

Let $S \subset TX$ be the vertical distribution; by $\Omega^k(S)$ we mean the sections of $\^ ^k S^*$.  Then there is a canonical isomorphism,
$$\Lambda^{n;k}(X) = \pi^*\Omega^n(B) \tens \Omega^k(S).$$
Each $\Lambda^{n;\bullet}(X)$ is a differential complex, for the fibrewise differential $d_S = 1 \tens d$, and thus we have a cohomology,
$$H^{n;k}_{d_S}(\Lambda) := \frac{\{\sigma \in \Lambda^{n;k}(X) \;|\; d_S \sigma = 0\}}{d_S\Lambda^{n;k-1}(X)}.$$

\begin{prop}\label{d_S fibre decomposition}
$H^{n;k}_{d_S}(X)$ is naturally isomorphic to the sections of a finite-dimensional vector bundle over $B$, whose fibre over $x \in B$ is
$$\left(\^ ^n T_x^*B\right) \tens H^k_{dR}\left(\pi^{-1}(x)\right).$$
\end{prop}

We give the correspondence, with proof omitted: a section over $B$ of $(\^ ^n T^*B) \tens H^k_{dR}(S)$ has representatives in each fibre.  These may be chosen smoothly, giving an element of $\Lambda^{n,k}(X)$.  Conversely, given $[\rho]_{d_S} \in H^{n;k}_{d_S}(X)$, let $\rho \in \Lambda^{n;k}(X)$ be a representative, and produce a section of $(\^ ^n T^*B) \tens H^k_{dR}(S)$ by taking, in each fibre, the $d_S$-cohomology of $\rho$.

\begin{rem}
Since $\Lambda^{n+1;k-1}(X) = F^{n+k}_{n+1}/F^{n+k}_{n+2}$ and $\Lambda^{n;k}(X) = F^{n+k}_n/F^{n+k}_{n+1}$, we get a short exact sequence of complexes,
$$0 \to \Lambda^{n+1;k-1}(X) \to F^{n+k}_n/F^{n+k}_{n+2} \to \Lambda^{n;k}(X) \to 0,$$
giving rise to a long exact sequence in cohomology.
\end{rem}

\begin{defn}
The \emph{Gauss-Manin connection},
$$\nabla : H^{n;k}_{d_S}(X) \to H^{n+1;k}_{d_S}(X),$$
is the connecting homomorphism arising from the short exact sequence
$$0 \to \Lambda^{n+1;k-1}(X) \to F^{n+k}_n/F^{n+k}_{n+2} \to \Lambda^{n;k}(X) \to 0.$$
\end{defn}

\begin{prop}
The Gauss-Manin connection is flat, i.e., $\nabla^2=0$;
\end{prop}

\subsection*{Computing the connection}

If we have chosen a horizontal distribution $N \subset TX$ complementary to $S$, then as in Lemma \ref{d decomposition}, we get a decomposition $d = \nabla + \Theta + d_S$.  Furthermore, we have $\Lambda^{n;k}(X) \iso \Omega^{n;k}(X)$.  Under this isomorphism, $d_S$ as defined in this section agrees with $d_S$ as defined in Section \ref{d decomposition}.  As remarked previously, the $\nabla$ of Lemma \ref{d decomposition} passes to a differential on $d_S$-cohomology, and agrees with the Gauss-Manin connection. 

We may complexify the above story, giving a decomposition
$$\C \tens H^{n;k}_{d_S}(X) = \bigoplus_{i+j=n} H^{i,j;k}_{d_S}(X),$$
where $i$ and $j$ are the holomorphic and anti-holomorphic degrees on $B$.  Passing to $d_S$-cohomology, $\del$ and $\bar\del$ give the holomorphic and anti-holomorphic parts of $\nabla$.

\begin{rem}
From equation \eqref{AC2} we see that, in $d_S$-cohomology, we have the anticommutation relation
\begin{equation}
\del\bar\del + \bar\del\del = 0 \label{AC5}
\end{equation}
\end{rem}

\subsection{Conditions on symplectic families}

We now give necessary (and in some cases sufficient) conditions for smooth symplectic families over complex bases to be generalized complex.  If the fibres or the base are 2-real-dimensional, we can usually give stronger statements.  In this section, we look only at condition (A) in Proposition \ref{H iff alpha and beta}.

\begin{thm}\label{[omega] pluriharmonic}
Let $\pi : X \to B$ be a smooth symplectic family over a complex manifold, with leafwise symplectic form $\omega$.

If these data are generalized complex, then
\begin{equation*}\label{dH4 cohomology}
\del\bar\del[\omega]_{d_S} = 0.
\end{equation*}
\end{thm}
In this case, we say that $[\omega]_{d_S}$ is \emph{pluriharmonic}.

\begin{proof}
If these data are generalized complex, then from condition (A) in Proposition \ref{H iff alpha and beta} we know that there exists $\alpha$ such that
$$i(\del\bar\del - \bar\del\del)\omega = d_S\alpha.$$
This is equivalent to the vanishing of $(\del\bar\del - \bar\del\del)\omega$ in $d_S$-cohomology.  Applying the anticommutation relation \eqref{AC5}, the result follows.
\end{proof}

In fact, since $\del\bar\del[\omega]_{d_S}=0$ is equivalent to the existence of a form solving condition (A), in certain circumstances it is sufficient for the existence of a generalized complex structure.

\begin{cor}\label{over Riemann surface}
If in the above setting $B$ is a Riemann surface, then $(P,I)$ are generalized complex if and only if $\del\bar\del[\omega]_{d_S} = 0$.
\end{cor}
\begin{proof}
Since the number of complex dimensions is 1, as per Remark \ref{type 1}, conditions (B) and (C) in Proposition \ref{H iff alpha and beta} are trivial.  Then the result follows from the above remark.
\end{proof}

If the smooth symplectic family has local symplectic trivializations, then $\nabla[\omega]_{d_S}=0$, which implies $\del\bar\del[\omega]_{d_S}=0$.  So we can say, for example,
\begin{cor}
A symplectic fibre bundle over a Riemann surface is generalized complex.
\end{cor}

Or we can generate counterexamples:
\begin{cor}\label{counterexamples 1}
Let $X \to B$ be a symplectic fibre bundle with compact fibres over a complex manifold $B$, with Poisson structure $P$ and transverse complex structure $I$..

If $V : B \to \R$ is a smooth, positive, function on the base, then the quotient $P/V$ is a Poisson structure on $X$.  But if $V$ is \emph{not} pluriharmonic then $(P/V,I)$ are not generalized complex.
\end{cor}
\begin{proof}
Let $\omega$ be the fibrewise symplectic form for $P$.  Then the fibrewise symplectic form for $P/V$ is $V\omega$.  So,
$$ \del\bar\del[V\omega]_{d_S} = (\del\bar\del V)[\omega]_{d_S} \neq 0$$
\end{proof}

If a smooth symplectic family $X \to B$ has 2-dimensional fibres, then $[\omega]_{d_S}$ just measures the symplectic volume of each fibre.  Thus it can be identified with a positive real function on the base.  But pluriharmonic \emph{functions} are well-understood, and the pluriharmonicity condition is quite strong:
\begin{thm}\label{V pluriharmonic}
Let $\pi:X \to B$ be a smooth symplectic family with compact, 2-dimensional fibres over a compact, connected complex manifold $B$.

If these data are generalized complex, then the function $V : B \to \R$ giving the symplectic volume of each fibre must be constant---in fact, $X \to B$ is a \emph{symplectic fibre bundle}.
\end{thm}

\begin{proof}
The claim that the symplectic volume is constant just follows from an application of the maximum principle for pluriharmonic functions on compact connected manifolds.  The second part of the claim, that $X \to B$ has local symplectic trivializations, is an application of Lemma \ref{symplectic bundle}, as follows.
\end{proof}

\begin{lem}\label{symplectic bundle}
Let $\pi:X \to B$ be a compact connected smooth symplectic family, with fibrewise symplectic form $\omega$.

If $[\omega]_{d_S}$ is flat under the Gauss-Manin connection, then $\pi:X \to B$ is a \emph{symplectic} fibre bundle for the symplectic form $\omega$. 
\end{lem}

\begin{proof}
Results like this are standard in the literature (see \cite{McDuffSalamon} for similar examples).  In this case, we only sketch the argument:

Given a smooth local trivialization about a point $x \in B$, one has a fibrewise symplectic form $\omega$ which may not be be \emph{constant} over $B$ in this trivialization, but, since $[\omega]_{d_S}$ is flat over $B$, $\omega$ is fibrewise cohomologous to a constant-over-$B$ form.  One then uses Moser stability to find fibrewise isotopies from the given local trivialization to this new \emph{symplectic} trivialization.  We note that this may be accomplishes smoothly over $B$.
\end{proof}

\begin{example}
We would like to give examples of generalized complex structures on smooth symplectic families which do not admit symplectic trivializations.  The following class of examples, which are noncompact, have nontrivial cases whose fibres are 2-dimensional.  (As per Theorem \ref{V pluriharmonic}, there will be no such examples of \emph{compact} surface bundles.)  Later we will give a compact example of higher rank.

Let $X \to B$ be a \emph{flat} symplectic fibre bundle over a noncompact Riemann surface $B$ with fibrewise symplectic form $\omega$, and suppose that $V : B \to \R$ is a nonconstant, positive pluriharmonic function (for example, a real linear function on the upper half plane $\sr H^+ \subset \C$).

Then $V\omega$ determines a Poisson structure, $P/V$, on $X$ and, as usual, the complex structure on $B$ pulls back to a transverse complex structure $I$ on $X$.  Since $V$ is not constant, the fibrewise symplectic volume is not costant, and thus $P/V$ does not admit symplectic trivializations.

\begin{prop}
In this case, $(P/V,I)$ are generalized complex.
\end{prop}
\begin{proof}
The given connection determines a decomposition of $d$ as in Lemma \ref{d decomposition}.  Since the connection is flat, the curvature $\Theta$ vanishes.  Since the connection is symplectic for $\omega$, we have that $\del\omega=\bar\del\omega=0$.  Then
\begin{eqnarray*}
\del\bar\del (V\omega) &=& (\del\bar\del V) \omega + (\del V)(\bar\del\omega) - (\bar\del V)(\del\omega) + V(\del\bar\del\omega) \\
&=& 0 + 0 - 0 + 0
\end{eqnarray*}
Then the claim follows from Corollary \ref{over Riemann surface}.
\end{proof}
\end{example}

\subsection{Higher-rank smooth symplectic families}

We now consider smooth symplectic families whose fibres may have dimension greater than 2.  In this case, the fibrewise 2nd cohomology, $H^{0,0;2}_{d_S}(X)$, is a vector bundle with rank possibly greater than 1; thus there is no maximum principle for its pluriharmonic sections.  In particular, we cannot say, even in the compact case, that if a smooth symplectic family of high rank is generalized complex then it is a symplectic fibre bundle (in disanalogy with Theorem \ref{V pluriharmonic}).  In fact, we provide a counterexample (Example \ref{counterexample})---a compact smooth symplectic family over a complex manifold which comes from a generalized complex structure but whose symplectic form is not flat in cohomology.

However, we can recover the existence of symplectic trivializations if we impose some topological conditions, of which we give a few examples.

\begin{prop}\label{omega flat on trivial bundle}
Let $\pi:X \to B$ be a smooth symplectic family over a compact connected complex manifold, with fibrewise symplectic structure $\omega$.  Furthermore, suppose $H^{0,0;2}_{d_S}(X)$ has a flat trivialization over $B$.

If these data are generalized complex, then they determine a symplectic fibre bundle.
\end{prop}
\begin{proof}
If these data are generalized complex, then by Theorem \ref{[omega] pluriharmonic} $[\omega]_{d_S}$ is pluriharmonic.  In the flat trivialization of $H^{0,0;2}_{d_S}(X)$, $[\omega]_{d_S} = f_1\sigma_1 + \ldots + f_k\sigma_k$, for flat basis sections $\sigma_i$ and functions $f_i$.   In this notation, the pluriharmonicity condition is just that each $f_i$ is pluriharmonic.  Since $B$ is compact, by the maximum principle this means that each $f_i$ is locally constant, and thus $[\omega]_{d_S}$ is flat.  The result follows from Lemma \ref{symplectic bundle}.
\end{proof}

We summarize the situation for two particular cases where the hypotheses of Propositions \ref{omega flat on trivial bundle} hold:
\begin{thm}\label{symplectic}
Let $\pi: X \to B$ be a smooth symplectic family over a compact complex manifold, which is generalized complex.  If $B$ is simply connected, or if $\pi: X \to B$ is a trivial bundle, then in fact $X$ is a symplectic fibre bundle over $B$.
\end{thm}
\begin{proof}
If $B$ is simply connected, then the Gauss-Manin connection trivializes $H^{0,0;2}_{d_S}(X)$, or if $\pi:X \to B$ is trivial, this induces a trivialization of $H^{0,0;2}_{d_S}(X)$.  In either case, the hypotheses of Proposition \ref{omega flat on trivial bundle} are satsified.
\end{proof}

\subsection{Generalized Calabi-Yau manifolds}\label{generalized Calabi-Yau manifolds}
A \emph{generalized Calabi-Yau} manifold (originally described in \cite{Hitchin2003}) is a generalized complex manifold whose canonical line bundle is generated by a global $d_H$-closed spinor.

Let $\pi:X \to B$ be a generalized complex smooth symplectic family over a complex manifold.  If $B$ is Calabi-Yau---that is, if its canonical bundle has a closed generating section $\rho_B$---then the spinor
$$\rho = e^{i\omega}\^\rho_B$$
on $X$ is $d_H$-closed for some closed 3-form $H$ and generates the canonical bundle $\kappa$ for the generalized complex structure.  Thus $X$ is generalized Calabi-Yau.  We note that Example \ref{counterexample} below is generalized Calabi-Yau in this way.

\begin{example}\label{counterexample}
In the higher-rank case, in contrast with surface bundles, the fact that a compact, connected smooth symplectic family is generalized complex does not imply that it will be a symplectic fibre bundle.  We give as a counterexample a certain generalized complex structure on a $T^4$-bundle over $T^2$.  ($T^k$ is the real $k$-dimensional torus.)

Consider the flat trivial bundle
$$X = T^4 \times \C \to \C.$$
Let $\theta_1,\theta_2,\theta_3,\theta_4$ be the standard basis of 1-forms for $T^4$, and let $x+iy$ be the complex coordinate on the base.  Let
$$\omega = \theta_1\^\theta_2 \,+\, \theta_3\^\theta_4 \,+\, x\,\theta_1\^\theta_3.$$
Let $N \subset TX$ be the horizontal distribution, giving a decomposition $d = d_S + \del + \bar\del$ and an extension of $\omega$ to $X$.  Then $d_S\omega = 0$ and $\del\bar\del\omega=0$, but $\nabla\omega \neq 0$---indeed, $\nabla[\omega]_{d_S} \neq 0$.

Let $\Lambda = \Z + i\Z \subset \C$ be the standard integral lattice.  We will define a monodromy homomorphism $\lambda : \Lambda \to \Aut(T^4)$ as follows: in the imaginary direction, $\lambda(i) = \Id$, and in the real direction, $\lambda(1)$ is the automorphism of $T^4$ which takes $\theta_2$ to $\theta_2 - \theta_3$ and leaves the others fixed.  Then
$$\lambda(1)^* : \omega \mapsto \theta_1\^\theta_2 + \theta_3\^\theta_4 + (x-1)\,\theta_1\^\theta_3.$$

Thus, at any $m+in \in \Lambda \subset \C$,
\begin{eqnarray*}
\left(\lambda(m+in)^*\omega\right)(m+in) &=& \theta_1\^\theta_2 + \theta_3\^\theta_4 + (m-m)\,\theta_1\^\theta_3 \\
&=& \omega(0)
\end{eqnarray*}
Then $\omega$ passes to $\tilde\omega$ on the flat bundle $\tilde{X} = X / \Lambda$.  It is still the case that $d_S\tilde\omega = 0$ and $\del\bar\del\tilde\omega=0$, so with the choices $\alpha=0$ and $\beta=0$ in Proposition \ref{H iff alpha and beta}, we see that these data come from a generalized complex structure.  But $[\tilde\omega]_{d_S}$ is still not flat, so $(\tilde{X},\tilde\omega)$ is not a symplectic fibre bundle.
\end{example}

\section{Higher complex type and the full obstruction}\label{the full obstruction}

So far, we have only studied condition (A) from Proposition \ref{H iff alpha and beta}.  For smooth symplectic families over Riemann surfaces, this condition, rephrased as a pluriharmonicity condition in cohomology, was both necessary and sufficient for the existence of a compatible generalized complex structure.  However, if the number of complex dimensions is 2 or more then conditions (B) and (C) may be nontrivial.

The technique we use (in Section \ref{the calculation}) is to try to solve (A), (B) and (C) in sequence.  Though we will be concrete, what we are doing in fact is working our way through a spectral sequence, with differential $d_S$ on the first page, $\nabla$ on the second and $\Theta$ (roughly speaking) on the third.  Figure ~\ref{fig:differentials} is suggestive---we can see that the targets of $d_S$, $\nabla$ and $\Theta$ form a diagonal, as would be expected in a spectral sequence coming from a bigrading.

We might hope that, as a sufficient condition, if the smooth symplectic family is in fact a symplectic fibre bundle then it is generalized complex.  This is not the case.  (See Example \ref{bad symplectic bundle} for a counterexample.)  However, a symplectic fibre bundle \emph{does} solve (A) and (B), and the remaining condition (C) can be understood as a cohomological obstruction on the base.

\subsection{The spectral sequence}\label{spectral sequence}

As we remarked earlier,
$$\nabla : [\sigma]_{d_S} \mapsto [\nabla\sigma]_{d_S}$$
determines a differential in $d_S$-cohomology.  Thus we may take the $\nabla$-cohomology of $H^\bullet(\Omega_T;d_S)$ to give $H^\bullet(\Omega_T;d_S,\nabla)$.

$H^\bullet(\Omega_T;d_S,\nabla)$ itself carries a differential, as follows: if $\mu$ is a form representing a class in $H^\bullet(\Omega_T;d_S,\nabla)$, then $d_S\mu=0$ and $\nabla[\mu]_{d_S} = 0$, i.e.,
$$\nabla\mu = d_S\nu$$
for some $\nu$.  One can check using the commutation relations that $\Theta\mu+\nabla\nu$ also determines a class in $H^\bullet(\Omega_T;d_S,\nabla)$.  We let
$$\Theta : [\mu]_{d_S,\nabla} \to [\Theta\mu + \nabla\nu]_{d_S,\nabla}.$$
Of course, without changing the class in $H^\bullet(\Omega_T;d_S,\nabla)$, we may replace $\mu$ with $\mu + \mu'$ and $\nu$ with $\nu + \nu' + \nu''$, where
$$\nabla\mu' = d_s\nu' \quad\textnormal{and}\quad d_s\nu''=0.$$
Then one can check that, nonetheless, $\Theta[\mu]_{d_S,\nabla} = \Theta[\mu']_{d_S,\nabla}$.  Furthermore, $\Theta^2[\mu]_{d_S,\nabla}=0$. Thus $\Theta$ defines a differential complex and we may pass from $H^\bullet(\Omega_T;d_S,\nabla)$ to $H^\bullet(\Omega_T;d_S,\nabla,\Theta)$.

We consider these differentials over the \emph{truncated complex} $\Omega^\bullet_T$ (as discussed in Section \ref{truncated complex}).  The first four pages of the spectral sequence are
$$\Omega^\bullet_T \to H^\bullet(\Omega_T;d_S) \to H^\bullet(\Omega_T;d_S,\nabla) \to H^\bullet(\Omega_T;d_S,\nabla,\Theta)$$

\subsection{The calculation}\label{the calculation}

We now attempt to solve in turn the conditions (A), (B) and (C) from Proposition \ref{H iff alpha and beta}.  Recall that, given the data of a Poisson structure with leafwise symplectic form $\omega$ and a transverse complex structure, we needed a real $(1,1;1)$-form $\alpha$ and a real $(2,1;0)+(1,2;0)$-form $\beta$ such that
$$\begin{array}{rccll}
\Phi^A(\omega) :=& i(\del\bar\del-\bar\del\del)\omega &=& d_S\alpha \quad& (A) \\
\Phi^B(\omega):=& i\left(\nabla(\theta_- - \theta_+) + \Theta(\bar\del - \del)\right)\omega &=& \nabla\alpha + d_S\beta \quad& (B) \\
\Phi^C(\omega):=& i\Theta(\theta_- - \theta_+)\omega &=& \Theta\alpha + \nabla\beta \quad& (C) 
\end{array}$$

\subsubsection*{Step A}
Let $[\Phi(\omega)]^A:=[\Phi^A(\omega)]_{d_S} \in H^{1,1;2}(\Omega;d_S)$.  Suppose that $[\Phi(\omega)]^A=0$.  Then we can solve (A), that is, there is an $\alpha$ such that 
$$i(\del\bar\del + \bar\del\del)\omega = d_S\alpha.$$
$\alpha$ is not fixed by $\omega$---we can replace it with any $\alpha + \alpha'$, where $d_S\alpha'=0$.

\subsubsection*{Step B}
Suppose that (A) is solvable.  Now from (B) we consider the term
$$\Phi^B(\omega) - \nabla\alpha = i\left(\nabla(\theta_- - \theta_+) + \Theta(\bar\del - \del)\right)\omega - \nabla\alpha.$$
It is clear that the corresponding class in $H^{3;1}(\Omega_T;d_S,\nabla)$ does not depend on the hoice of $\alpha$ solving (A).  We call this class $[\Phi(\omega)]^B$.

If $[\Phi(\omega)]^B=0$, then there is some $[\alpha'] \in H^{1,1;1}(\Omega;d_S)$ such that
$$[\Phi^B(\omega)-\nabla\alpha]_{d_S}=\nabla[\alpha']_{d_S},$$
i.e.,
$$\Phi^B(\omega) - \nabla\alpha =
\nabla\alpha' + d_S\beta$$
for some $(2,1;0)+(1,2;0)$-form $\beta$.  Then $\alpha+\alpha'$ and $\beta$ solve conditions (A) and (B).

\subsubsection*{Step C}
Suppose that (A) and (B) are solvable.  Now from $(C)$ we consider the term
$$\Phi^C(\omega) - \Theta\alpha - \nabla\beta \;=\; i\Theta(\theta_- - \theta_+)\omega - \Theta\alpha - \nabla\beta.$$
One can check that the corresponding class in the triple cohomology $H^{4;0}(\Omega_T;d_S,\nabla,\Theta)$ does not depend on the choice of $\alpha$ and $\beta$ solving (A) and (B).  We call this class $[\Phi(\omega)]^C$.

If $[\Phi(\omega)]^C=0$, then there is some $[\alpha'] \in H^{1,1;1}(\Omega_T;d_S,\nabla)$ such that
$$[\Phi^C(\omega) - \Theta\alpha - \nabla\beta]_{d_S,\nabla} = \Theta[\alpha']_{d_S,\nabla}.$$
If $\nabla\alpha' = d_S\beta'$, then this just means that
$$\left[\Phi^C(\omega) - \Theta\alpha - \nabla\beta\right]_{d_S,\nabla} = [\Theta\alpha' + \nabla\beta']_{d_S,\nabla},$$
i.e.,
$$\left[\Phi^C(\omega) - \Theta\alpha - \nabla\beta\right]_{d_S} = [\Theta\alpha' + \nabla\beta']_{d_S} + \nabla[\beta'']_{d_S}$$
for some $\beta''$.  Since this equation is basic (i.e., of degree $(4;0)$), and the basic, $d_S$-exact forms are trivial, we can drop the $d_S$-cohomology, and we have
$$\Phi^C(\omega) - \Theta\alpha - \nabla\beta = \Theta\alpha' + \nabla\beta' + \nabla\beta''.$$
Thus $\alpha + \alpha'$ and $\beta + \beta' + \beta''$ solve (C), and we can check that they still solve (A) and (B).

Conversely, if $\alpha$ and $\beta$ solve (A), (B) and (C), it is trivial that the relevant cohomology classes vanish.  Therefore,

\begin{thm}\label{cohomology condition}
Suppose that $P$ and $I$ are a regular Poisson structure and a transverse complex structure respectively, with leafwise symplectic form $\omega$.  Let $\Phi^A(\omega)$, $\Phi^B(\omega)$ and $\Phi^C(\omega)$ be the obstruction forms as in Definition \ref{obstructions}, and let
\begin{align*}
[\Phi(\omega)]^A &\in H^{2;2}(\Omega;d_S) \\
[\Phi(\omega)]^B &\in H^{3;1}(\Omega_T;d_S,\nabla) \\
[\Phi(\omega)]^C &\in H^{4;0}(\Omega_T;d_S,\nabla,\Theta)
\end{align*}
be as defined above.

Then $(P,I)$ are generalized complex if and only if $[\Phi(\omega)]^A$, $[\Phi(\omega)]^B$ and $[\Phi(\omega)]^C$ all vanish.
\end{thm}

\begin{rem}
Strictly speaking, $[\Phi(\omega)]^B$ is only well-defined if $[\Phi(\omega)]^A=0$, and $[\Phi(\omega)]^C$ is only well-defined if $[\Phi(\omega)]^B=0$.
\end{rem}

\begin{prop}\label{A and B if symplectic}
Let $X \to B$ be a symplectic fibre bundle over a complex base, with leafwise symplectic form $\omega$.  Then $[\Phi(\omega)]^A=0$ and $[\Phi(\omega)]^B=0$.
\end{prop}

\begin{proof}
Since $X \to B$ has symplectic trivializations, $\nabla[\omega]_{d_S}=0$, thus there exists $\gamma \in \Omega^{0,1;1}(X)$ such that $\nabla\omega = -d_S(\gamma + \bar\gamma)$.  In particular, $\bar\del\omega=-d_S\gamma$ and $\del\omega=-d_S\bar\gamma$.  Then
\begin{align*}
\Phi^A(\omega) &= i(\del\bar\del - \bar\del\del)\omega \\
&= i(\del\bar\del + \bar\del^2 - d_S\theta_- - \bar\del\del - \del^2 + d_S\theta_+)\omega
\qquad\textnormal{(by relation \eqref{AC1})} \\
&= i\nabla(\bar\del\omega - \del\omega) + id_S(\theta_+ - \theta_-)\omega \\
&= id_S\left(\nabla(\gamma-\bar\gamma)+(\theta_- - \theta_+)\omega\right) 
\end{align*}
So $\alpha = i\nabla(\gamma-\bar\gamma)+i(\theta_- - \theta_+)\omega$ solves (A).

Given this choice of $\alpha$, we apply the relation $\nabla^2 = -d_S\Theta - \Theta d_S$ as follows:
\begin{align*}
\nabla\alpha &= i\nabla^2(\gamma-\bar\gamma) + i\nabla(\theta_- - \theta_+)\omega \\
&= -id_S\Theta(\gamma-\bar\gamma) - i\Theta d_S(\gamma-\bar\gamma) + i\nabla(\theta_- - \theta_+)\omega \\
&= -id_S\Theta(\gamma-\bar\gamma) + i\Theta (\bar\del - \del\bar)\omega + i\nabla(\theta_- - \theta_+)\omega
\end{align*}
If we let $\beta = i\Theta(\gamma-\bar\gamma)$ then
\begin{align*}
\Phi^B(\omega) &= i\nabla(\theta_- - \theta_+)\omega + i\Theta(\bar\del-\del)\omega \\
&= \d_S\beta + \nabla\alpha
\end{align*}
and thus (B) is also solved.
\end{proof}

\begin{rem}\label{C is basic}
So in a natural class of examples, two of the three conditions are satisfied; and condition (C) is not so bad after all: the form $\Phi^C(\omega) - \Theta\alpha - \nabla\beta$ which determines the obstruction is $d_S$-closed and of type $(4;0)$---that is, it is \emph{basic}.  Thus, in the case of a smooth symplectic family $X\to B$, the obstruction is in the cohomology of the base.  During the calculation, we get the term
$$[\Phi^C(\omega)-\Theta\alpha-\nabla\beta]_{d_S,\nabla} \in H^4\left(\Omega_T(B);d\right).$$
If $B$ is compact then this cohomology is finite-dimensional.  Then the final step---taking the $\Theta$-cohomology---is a finite-dimensional calculation.  (See Example \ref{bad symplectic bundle}.)
\end{rem}

So in general, for a complex manifold $B$, if $H^4(\Omega_T(B);B)=0$ then any symplectic fibre bundle over $B$ will be generalized complex.

Furthermore, if $B$ is K\"ahler, we can move from the truncated de Rham cohomology to the usual de Rham cohomology.  To be precise, a class in $H^k(\Omega_T(B);d)$ determines a class in $H^k(\Omega(B);d)$, and we can say
\begin{lem}\label{kahler not truncated}
If $B$ is K\"ahler, or satisfies the $\del\bar\del$-lemma, then the map
$$H^k(\Omega_T(B);d) \to H^k(\Omega(B);d)$$
is injective for $k>2$.
\end{lem}
This is a well-known type of result.  For example, the author learned the technique of proof from \cite{Goto}.
\begin{proof}
Suppose that $[\alpha]_T \in H^k(\Omega_T(B);d)$, i.e., $\alpha$ is a real $d$-closed $k$-form with no components of degree $(k,0)$ or $(0,k)$.  Let $[\alpha]$ be the corresponding class in the usual $H^k(\Omega(B);d)$, and suppose that $[\alpha]=0$, i.e., there is a $k-1$-form $\beta$ such that $d\beta=\alpha$.

Consider a degree of $d\alpha=0$:
\begin{align*}
(d\alpha)^{k,1} &= \del\alpha^{k-1,1} + \bar\del\alpha^{k,0} \\
0 &= \del\alpha^{k-1,1} + 0
\end{align*}
And a degree of $\alpha=d\beta$:
\begin{align*}
\alpha^{k-1,1} &= \del\beta^{k-2,1} + \bar\del\beta^{k-1,0} \\
\alpha^{k-1,1} - \del\beta^{k-2,1} &= \bar\del\beta^{k-1,0}
\end{align*}
So $\alpha^{k-1,1} - \del\beta^{k-2,1}$ is $\bar\del$-exact and $\del$-closed.  By the $\del\bar\del$-lemma, this implies that for some $\gamma \in \Omega^{k-2,0}(B)$,
$$\bar\del\beta^{k-1,0} = \del\bar\del\gamma.$$
Let $\beta' = \beta - \beta^{k-1,0} + \bar\del\gamma$.  Then $\beta'$ has no $(k-1,0)$ component, and
\begin{align*}
d\beta' &= d\beta - \del\beta^{k-1,0} - \bar\del\beta^{k-1,0} + \del\bar\del\gamma \\
&= d\beta - 0 - \del\bar\del\gamma + \del\bar\del\gamma \\
&= \alpha
\end{align*}
So $\alpha$ has a primitive with no component in degree $(k-1,0)$.  Similarly, we may remove the component in $(0,k-1)$, and thus $\alpha$ has a primitive in the truncated complex $\Omega_T^\bullet(B)$.  I.e., since $[\alpha]=0$, therefore $[\alpha]_T=0$.
\end{proof}

\begin{cor}
If $B$ is K\"ahler and $H^4(B)=0$, then any symplectic fibre bundle $X\to B$ is generalized complex.
\end{cor}

The above result is not very useful in low dimensions.  If, for example, we are considering a symplectic fibre bundle over a K\"ahler manifold, whose total space has real dimension 6, then the nontrivial case is when the fibres are symplectic surfaces and the base is a complex K\"ahler surface.  But then $H^4(B)$ just measures the volume of a volume form on the base, and is nontrivial.  To determine whether the generalized complex structure is obstructed, we must actually calculate $[\Phi(\omega)]^C$ by integrating $\Phi^C(\omega)$ over the base, as in the following:

\begin{example}\label{bad symplectic bundle}
We give an example of the failure of condition (C), a compact symplectic fibre bundle which is not generalized complex.
\begin{rem}
The version of Example \ref{bad symplectic bundle} in the original published version of this paper contained an error, making it fail as an example of this phenomenon.  This was pointed out to us by Gil Cavalcanti, who earlier proved \cite[Theorem 2.3]{Cavalcanti} that a symplectic surface bundle over a generalized complex base whose fibre is not a torus is indeed generalized complex.  The error in the argument for our original example (constructed as the projectivization of the tangent bundle of $\C P^2$) was an invalid inference from known facts about the curvature of the Fubini-Study metric to conclusions about the relevant $\theta_\pm$.  We provide a new and correct example of the phenomenon.  We would like to thank Marco Gualtieri for discussions which greatly clarified this example.
\end{rem}

Let $B$ be a compact complex Kahler surface for which there exist $a,b \in H^{2,0 + 0,2}(B,\Z)$ such that $a \^ b \neq 0$.  For example, $B$ could be $T^4 = \C^2 / \Z^4$, with $a = b = [dz_1\^dz_2 + d\bar{z}_1 \^ d\bar{z}_2]$.  Then $a$ and $b$ are the Chern classes of $S^1$ bundles $X_a \to B$ and $X_b \to B$ respectively, whose fibres come with standard invariant $1$-forms $dt_a$ and $dt_b$ respectively.  The symplectic bundle we use as our example will be
$$X := X_a \x_B X_b \to B,$$
with symplectic form $\omega = dt_a \^ dt_b$.

We may choose $S^1$-equivariant connections on each of $X_a$ and $X_b$, giving a $T^2$-equivariant connection on $X \to B$.  Then, by construction, $\theta_+$---the $(2,0)$ part of the curvature---contracts with $dt_a$ to give the class $a^{2,0}$ and with $dt_b$ to give the class $b^{2,0}$ (and similarly for $\theta_-$ and $a^{0,2}$ and $b^{0,2}$).  Note that, since the Chern class has no $(1,1)$ part, $\theta_0, dt_a$ and $\theta_0, dt_b$ integrate to $0 \in H^2(B)$.

As per Proposition \ref{A and B if symplectic}, conditions (A) and (B) are satisfied for some $\alpha$ and $\beta$.  In fact, since $\omega$ is flat under the given connection, i.e., $\del\omega=\bar\del\omega=0$, $\alpha=0$ and $\beta=0$ solve (A) and (B).

Given $\alpha=0$ and $\beta=0$, the remaining obstruction is
$$\Phi^C(\omega) = i\Theta(\theta_- - \theta_+)\omega$$
As per Remark \ref{C is basic} and Lemma \ref{kahler not truncated}, $H^{4;0}(\Omega_T;d_S,\nabla)$ is naturally contained in $H^4(B)$, and so we compute
\begin{align*}
[i(\theta_+\theta_- - \theta_-\theta_+)\omega]_{d_S,\nabla} &= [2i\theta_+dt_a \^ \theta_-dt_b - 2i\theta_+dt_b \^ \theta_-dt_a]_{d_S,\nabla} \\
&= 2i\, a^{2,0} \^ b^{0,2} - 2i\, b^{2,0} \^ a^{0,2} \;\in H^4(B)  \\
&= 2i\, a \^ b \neq 0.
\end{align*}

Of course, before we can know if the generalized complex structure is obstructed, we must pass to $H^{4;0}(\Omega_T;d_S,\nabla,\Theta)$.  Does there exist a $[c] \in H^{2;1}(\Omega_T;d_S,\nabla)$ such that $\Theta [c] = 2i\, a \^ b$?  Since we are dealing with the \emph{truncated} complex, $[c]$ must be of type $(1,1;1)$.  By dimension count, $\theta_+[c]=0$ and $\theta_-[c]=0$.  But, as we mentioned above, $\theta_0$ vanishes in cohomology, so the answer is no.  Then $[\Phi(C)]_{d_S,\nabla,\Theta}$ does not vanish and there exists no compatible generalized complex structure.
\end{example}


\begin{thebibliography}{9}
\bibitem{Bailey}Bailey, Michael \emph{On the local and global classification of generalized complex structures}, Ph.D. thesis, University of Toronto, (2012).

\bibitem{BGV1992}Berline, Nicole; Getzler, Ezra; Vergne, Mich\`ele. \emph{Heat kernels and Dirac operators.} Springer, 1992.

\bibitem{CannasDaSilva}Cannas da Silva, Ana. \emph{Lectures on Symplectic Geometry.} Springer, 2001.

\bibitem{Cavalcanti}Cavalcanti, Gil R. \emph{New aspects of the $dd^c$-lemma}, Oxford University D. Phil thesis (2004).

\bibitem{Gualtieri2011}Gualtieri, Marco. \emph{Generalized complex geometry.} Ann. of Math. (2) \textbf{174} (2011), no. 1, 75--123.

\bibitem{Goto}Goto, Ryushi. \emph{Deformations of generalized complex generalized Kahler structures.} J. Differential Geom. \textbf{84} (2010), no. 3, 525--560.

\bibitem{Hitchin2003}Hitchin, Nigel. \emph{Generalized Calabi-Yau Manifolds.} Q. J. Math. \textbf{54} (2003), no. 3, 281--308.

\bibitem{KatzOda}Katz, Nicholas M.; Oda, Tadao. \emph{On the differentiation of De Rham cohomology classes with respect to parameters.} J. Math. Kyoto Univ. \textbf{8} (1968), no. 2, 199--213.


\end{thebibliography}
\end{document}